\documentclass[11pt]{amsart}
\usepackage{amssymb,amscd,amsmath}
\usepackage[all]{xy}
\newtheorem{thm}{Theorem}

\newtheorem{cor}[thm]{Corollary}

\newtheorem{prop}[thm]{Proposition}

\newtheorem{lem}[thm]{Lemma}
\theoremstyle{remark}

\theoremstyle{definition}

\renewcommand{\bar}{\overline}

\newcommand{\Sym}{\mathrm{Sym}}
\newcommand{\Id}{\mathrm{Id}}

\newcommand{\A}{{\mathbb{A}}}

\newcommand{\F}{{\mathbb{F}}}

\newcommand{\uX}{{\underline{X}}}
\newcommand{\uY}{{\underline{Y}}}
\newcommand{\uZ}{{\underline{Z}}}

\newcommand{\Gal}{\mathrm{Gal}}

\newcommand{\Spec}{\mathrm{Spec}\;}

\newcommand{\Alt}{{\raise 2pt\hbox{$\scriptstyle\bigwedge$}}}

\newcommand{\e}{\epsilon}
\newcommand{\cd}{\mathrm{cd}}

\begin{document}
\title[Residually finite dimensional algebras and almost identities]
{Residually finite dimensional algebras and polynomial almost identities}

\author{Michael Larsen}
\email{mjlarsen@indiana.edu}
\address{Department of Mathematics\\
    Indiana University \\
    Bloomington, IN 47405\\
    U.S.A.}

\author{Aner Shalev}
\email{shalev@math.huji.ac.il}
\address{Einstein Institute of Mathematics\\
    Hebrew University \\
    Givat Ram, Jerusalem 91904\\
    Israel}


\thanks{ML was partially supported by NSF grant DMS-1702152.
AS was partially supported by ISF grant 686/17 and the Vinik Chair of mathematics which he holds.
Both authors were partially supported by BSF grant 2016072.}

\begin{abstract}
Let $A$ be a residually finite dimensional algebra (not necessarily associative) over a field $k$.
Suppose first that $k$ is algebraically closed. We show that if $A$ satisfies a homogeneous almost identity $Q$,
then $A$ has an ideal of finite codimension satisfying the identity $Q$.
Using well known results of Zelmanov, we conclude that, if a residually finite dimensional Lie algebra $L$ over
$k$ is almost $d$-Engel, then $L$ has a nilpotent (resp. locally nilpotent) ideal of finite codimension if
\linebreak
 char $k=0$ (resp. char $k > 0$).

Next, suppose that $k$ is finite (so $A$ is residually finite). We prove that, if $A$ satisfies a homogeneous probabilistic identity $Q$, then $Q$ is a coset identity of $A$. Moreover, if $Q$ is multilinear, then $Q$ is
an identity of some finite index ideal of $A$.

Along the way we show that, if $Q\in k\langle x_1,\ldots,x_n\rangle$ has degree $d$, and $A$ is a finite $k$-algebra such that the probability that $Q(a_1, \ldots , a_n)=0$ (where $a_i \in A$ are randomly chosen) is at least $1-2^{-d}$, then $Q$ is an identity of $A$.
This solves a ring-theoretic analogue of a (still open) group-theoretic problem posed by Dixon,
\end{abstract}

\maketitle

\section{Introduction}

In this paper, we prove three theorems concerning residually finite dimensional algebras $A$ and polynomial identities.
The common theme is that if a (non-commutative) homogeneous polynomial $Q$ in $n$ variables vanishes on a
large enough subset of $A^n$, then it is actually a \emph{coset identity}, that is, it holds identically on
$(a_1+I)\times\cdots\times (a_n+I)$ for some two-sided ideal $I$ of finite codimension in $A$
and some $a_1,\ldots,a_n\in A$. Under some assumptions we obtain stronger conclusions, namely, that $Q$
is an identity of the ideal $I$.

Let $k$ be an algebraically closed field, $V$ a $k$-vector space, possibly of infinite dimension, and $n$ a positive integer.  We recall \cite{LS} that the \emph{codimension} of a subset $X\subset V^n$ is the smallest integer $c$ for which there exists a direct sum decomposition of $k$-vector spaces $V^n = V_1\oplus V_2$, where $V_2$ is finite dimensional, and an algebraic set $X_2$ of codimension $c$ in $V_2$, such that $X\supset V_1\times X_2$.  We say that $X$ is of \emph{infinite codimension} if no such decomposition exists.

Let $A$ be an associative $k$-algebra, possibly non-unital.
Each non-commutative polynomial $Q\in k\langle x_1,\ldots,x_n\rangle$
defines the evaluation map $e_Q \colon A^ n\to A$.
We define $\cd_Q A$ to be the codimension of $e_Q^{-1}(0)$.
We say that $Q$ is an \emph{almost identity} if $\cd_Q A < \infty$.
If $\cd_Q A = 0$, or, equivalently,  $e_Q(A^n) = 0$, we say $Q$ is an \emph{identity} for $A$.

We can likewise consider a Lie (resp. Jordan) algebra $A$ over $k$ and a Lie (resp. Jordan) polynomial $Q$ and define the codimension of the zero set of $Q$ in $A^n$ and an almost identity in the analogous way. In fact the same definition applies for an arbitrary algebra $A$, namely a linear space over $k$ with a bilinear map $A \times A \to A$ as multiplication (possibly but not necessarily satisfying
some extra-conditions). In this case the polynomial $Q$ is an element of a free algebra in the respective category.  Note that an \emph{ideal} of $A$ will always mean a two-sided ideal, and $A$ \emph{residually finite dimensional} means that the intersection of all ideals of $A$ of finite codimension is the zero ideal.

Our first main result is the following:

\begin{thm}
\label{main}
Let $A$ be a residually finite dimensional algebra over $k$ and $Q$ a homogeneous polynomial as above.  Then the following are equivalent:
\begin{enumerate}
\item The polynomial $Q$ is an almost identity for $A$.
\item The polynomial $Q$ is an identity for some ideal $I$ of $A$ of finite codimension.
\end{enumerate}
\end{thm}

The non-trivial part is that (1) implies (2).  The reverse implication follows from the fact that $I^n$ is of finite codimension in $A^n$.

We now compare the theorem above with the main result of \cite{LS} (see Theorem 1 there and the comments on p. 10).
The latter result shows that an associative/Lie/Jordan algebra with an almost identity $Q$ satisfies some identity $P$
(usually different and more complex than $Q$).
Theorem \ref{main} above holds for \emph{all} algebras $A$, and the identity satisfied by the ideal $I$ is the original almost identity $Q$.

Our second main result concerns the case that $k$ is a finite field $\F_q$ for a prime power $q$.
Any residually finite algebra $A$ over $\F_q$ is a dense subalgebra of its completion $\bar A$, and
as an additive group, $\bar A$ is profinite.  For each $n$, we endow $\bar A^n$ with its Haar measure and consider the condition on a subset $X\subset A^n$ that $\bar X$ has non-zero measure.  We say that $Q$ is a \emph{probabilistic identity} if the closure of ${e_Q^{-1}(0)}$ has positive measure.

\begin{thm}
\label{finite fields}
Let $A$ be a residually finite dimensional algebra over $k=\F_q$ and $Q$ a homogeneous polynomial as above.  Then the following are equivalent:
\begin{enumerate}
\item The polynomial $Q$ is a probabilistic identity of $A$.
\item The polynomial $Q$ is a coset identity of some finite index ideal $I$ of $A$.
\end{enumerate}
Furthermore, if $Q$ is multilinear, then these conditions are equivalent to

{\rm (3)} $Q$ is an identity of some finite index ideal $I$ of $A$.
\end{thm}

Again, the non-trivial part is the claim that (1) implies (2).

In \cite{D} Dixon asks whether, for every group-word $w \in F_n$ (the free group of rank $n$)
there exists $\e = \e(w) > 0$ such that if $G$ is a finite group, and the word map $w:G^n \to G$
attains the value $1$ with probability $\ge 1-\e$, then $w$ is an identity of $G$.
In spite of some positive results for special words $w$, Dixon's Problem is still very much open.
Here we obtain a general positive solution of an analogous question on finite algebras $A$.
The solution is effective in the sense that $\e$ is given explicitly; in fact, if $d$ is the degree of the ambient polynomial map, then $\e = 2^{-d}$ will do.

\begin{thm}\label{dixon}
Let $A$ be any finite-dimensional algebra over a finite field and $e_Q \colon A^n\to A$ a polynomial map
associated with a polynomial $Q$ of degree $d$ in $n$ variables in the respective category.
If $\frac{|e_Q^{-1}(0)|}{|A|^n} \ge 1 - 2^{-d}$, then $Q$ is identically zero.
\end{thm}

For $d \ge 2$ consider the degree $d$ Engel polynomial $E_{d-1} := [x,y, \ldots , y]$, a left-normed Lie product where $y$ appears $d-1$ times, as an element of the free Lie algebra on $x,y$ over the underlying field $k$.
By \cite[2.1]{MM}, if $L$ is a finite Lie algebra in which the Engel condition $E_{d-1}=0$ holds with probability
greater than $1-2^{-d}$, then $E_{d-1}$ is an identity of $L$. Theorem \ref{dixon} above extends this for any
finite algebra and any polynomial.

\section{Algebras over algebraically closed fields}

We recall  that every subset $S$ of $k^m = \A^m(k)$ defines the ideal $Z(S)$ of elements in
$k[x_1,\ldots,x_m]$ which vanish on $S$ and every ideal $I\subset k[x_1,\ldots,x_m]$
defines the algebraic set $V(I)$ of common zeroes of $I$.  By Hilbert's Nullstellensatz, these two maps
give a bijection between algebraic sets in $k^n$ and radical ideals $I$.
Radical ideals $I$ are in bijective correspondence with reduced closed $k$-subschemes
$\Spec k[x_1,\ldots,x_m]/I$.  The bijection between algebraic sets and reduced closed $k$-subschemes
is given in one direction by taking the Zariski closure with its reduced closed subscheme structure and in the other by taking $k$-points.

Each linear transformation $T\colon V\to W$ of finite dimensional $k$-vector spaces defines a homomorphism of commutative graded $k$-algebras $T^*\colon \Sym^* W^*\to \Sym^* V^*$.  If $S_V\subset V$,
$S_T\subset W$, and $T(S_V)\subset S_W$, then $T^*(Z(S_W))\subset Z(S_V)$.
Thus, $T$ determines a morphism of affine $k$-schemes
$$\Spec \Sym^* V^*/Z(S_V)\to \Spec \Sym^* W^*/Z(S_W)$$
which, at the level of $k$-points, gives the restriction of $T$.

We now prove the key proposition:

\begin{prop}
\label{Super}
If $A$ is an arbitrary algebra over $k$,
$I$ is an ideal of $A$ of finite codimension,
and $Q$ is a homogeneous polynomial as above,
then
$$\cd_Q A \ge \cd_Q I + \cd_Q A/I.$$
\end{prop}

\begin{proof}
There is something to check only if $Q$ is an almost identity for $A$.
Let $V_1\oplus V_2$ be a direct sum decomposition of $A^n$ such that $e_Q^{-1}(0)$
contains $V_1\times X_2$ for some algebraic set $X_2$ of codimension $\cd_Q A$ in the finite-dimensional space $V_2$.  Let $V'_1 = V_1 \cap I^n$, and let $V^{\prime\prime}_1$ denote a complementary subspace to $V'_1$ in $V_1$.  Let $V'_2 = V_2 \oplus V^{\prime\prime}_1$
and $X'_2 = X_2 \times V^{\prime\prime}_1$.  Then $X'_2$ is of codimension $\cd_Q A$ in $V'_2$, and $V'_1\times X'_2 = V_1\times X_2\subset e_Q^{-1}(0)$.  Replacing $V_1,V_2,X_2$
by $V'_1,V'_2,X'_2$, we may therefore assume that $V_1 \subset I^n$.

We identify $V_2$ with the $k$-points of the variety $\A^{\dim V_2}$.  Let $\uX_2$ denote the Zariski closure of $X_2$ in $\A^{\dim V_2}$, so we can identify $X_2$ with $\uX_2(k)$, and
$\dim \uX_2 = \dim V_2 - \cd_Q A$.
We identify $(A/I)^n$ with the $k$-points of $\A^{n\dim A/I}$
and denote by $\uY$ the Zariski closure of the algebraic set
$$Y = \{(\bar a_1,\ldots,\bar a_n)\in (A/I)^n\mid e_Q(\bar a_1,\ldots,\bar a_n) = 0\},$$
so the algebraic set is identified with $\uY(k)$.

The projection map $A^n/V_1\to A^n/I^n$ maps the algebraic set $X_2$ to the algebraic set $Y$,
and it follows that the associated projection morphism $\A^{\dim V_2} \to \A^{n \dim A/I}$
defines a morphism $\pi\colon \uX_2\to\uY$.  As $0\in \uX_2(k)=X_2$ maps by $\pi$ to  $0\in \uY(k) = Y$, the fiber $\uZ$ of $\pi$ over $0$ is non-empty.
Now, $\uZ(k)$ is $X_2\cap (I^n/V_1)$, so $\cd_Q I \le \dim (I^n/V_1) - \dim \uZ$.
By \cite[Tag 02JS]{Stacks},
$$\dim \uX_2 \le \dim \uY + \dim \uZ.$$
Thus,
\begin{align*}
\cd_Q A &= \dim A^n/V_1 - \dim \uX_2 = \dim A^n/I^n + \dim I^n/V_1 - \dim \uX_2 \\
              &\ge \dim A^n/I^n - \dim \uY + \dim I^n/V_1 - \dim \uZ \ge \cd_Q (A/I) + \cd_Q I.
\end{align*}
\end{proof}

We can now prove Theorem~\ref{main}.

\begin{proof}
It suffices to prove that, if condition (2) does not hold, then for all $i\ge 0$ there exists an ideal $I$ of finite codimension in $A$
such that $\cd_Q A/I \ge i$.  We proceed by induction on $i$, the statement being trivial for $i=0$. If
the induction hypothesis holds,  as $Q$ is not an identity for $I$, there exists $\alpha\in I^n$ with $e_Q(\alpha)\neq 0$.  As $A$ is residually finite dimensional, there exists an ideal $J$ of $A$ of finite codimension such that $e_Q(\alpha)\not\in J$.  If $\bar\alpha$ denotes the image of $\alpha$
in $(I/I\cap J)^n$, then $e_Q(\bar\alpha)\neq 0$, so $Q$ is not an identity for $I/I\cap J$.  By Proposition~\ref{Super},
$$\cd_Q A/(I\cap J) \ge \cd_Q A/I + \cd_Q I/(I\cap J) \ge i+1,$$
and the theorem follows by induction.
\end{proof}

We now discuss some consequences of the theorem.

\begin{cor}
Let $k$ be a field of characteristic $0$, $A$ a finitely generated algebra over $k$,
and $Q$ a homogeneous polynomial in $n$ variables defined over $k$.  If $V$ is a subspace of $A^n$ of finite
codimension, $\vec a = (a_1,\ldots,a_n)\in A^n$, and $e_Q(\vec a + V) = 0$, then
$Q$ is an identity on an ideal of $A$ of finite codimension.
\end{cor}

\begin{proof}
As $A\otimes_k \bar k$ is finitely generated over $\bar k$, it is residually finite-dimensional.
As $k$ is infinite, if $Q$ is an identity on $\vec a+V$, then it is an identity on $\vec a+V\otimes_k \bar k$,
which is a subset of finite codimension in $A\otimes_k \bar k$.  Therefore, there exists an ideal $I$ of
finite codimension in $A\otimes_k \bar k$ such that $Q$ is an identity on $I$.

Let $\{a_1,\ldots,a_m\}$ be a generating subset for $A$ over $k$.  Let $\bar a_i$ denote
the image of $a_i\otimes 1$ under the quotient homomorphism
$$A\otimes_k \bar k\to (A\otimes_k \bar k)/I.$$
The $\bar k$-algebra $(A\otimes_k \bar k)/I$ is finite-dimensional, so fixing any basis, the structure
constants lie in a finite extension $K$ of $k$.  We fix a finite-dimensional $K$-algebra $B$
and an isomorphism of $K$-algebras $\iota\colon B\otimes_K \bar k\to (A\otimes_k \bar k)/I$.  There exists a finite extension $L/K$ such that $\iota^{-1}(\bar a_i)\in B\otimes_k L$ for all $i$.
Enlarging $L$, we may assume $L/K$ is Galois.  There is a unique $L$-algebra homomorphism
$\phi\colon A\otimes_k L\to B\otimes_K L$ such that the diagram
$$\xymatrix{A\otimes_k L\ar[d]_\phi\ar[r]&A\otimes_k \bar k\ar[r] &(A\otimes_k \bar k)/I \\
B\otimes_K L\ar[rr]&&B\otimes_K \bar k\ar[u]_\iota}$$
commutes.

We consider the $L$-algebra homomorphism
$$A\otimes_k L \to \bigoplus_{\sigma\in \Gal(L/k)} B\otimes_K L$$
which in the $\sigma$-coordinate is given by $(\Id_B\otimes \sigma)\circ \phi$.  The kernel is
invariant under the action of $\Gal(L/k)$ on $A\otimes_k L$, so
by Galois descent for vector spaces \cite[Chap.~5,~\S10~Prop.~6]{B} it is of the form
$W\otimes_k L$, where $W$ is the kernel of the composition of $A\to A\otimes_k L$
and $\phi$.  Thus $W$ is a finite codimension ideal of $A$.

\end{proof}

\begin{cor}
\label{nagata}
Let $A$ be a residually finite dimensional associative algebra over $k$. Let $d \ge 1$ and suppose
$x^d$ is an almost identity for $A$. If $k$ has characteristic $p>0$ suppose also $p > d$.
Then $A$ has an ideal $I$ of finite codimension satisfying $I^{f(d)}=0$, where $f(d)$ is a suitable
function of $d$.
\end{cor}

\begin{proof}
By Theorem \ref{main}, $A$ has an ideal $I$ of finite codimension satisfying the identity $x^d=0$.

The well known Nagata-Higman Theorem applied for the associative (non-unital) $k$-algebra $I$
shows that $I^{f(d)}=0$. See for instance \cite[Chapter 6]{DF} for the theorem and for explicit
bounds on the function $f$.
\end{proof}

Our next result describes almost $d$-Engel Lie-algebras.

\begin{thm}
\label{engel}
Let $L$ be a residually finite dimensional Lie algebra over $k$. Let $d \ge 1$ and suppose
The Engel polynomial $E_d$ is an almost identity for $L$. Then
\begin{enumerate}
\item If $k$ has characteristic zero then $L$ has a nilpotent ideal of finite codimension.
\item If $k$ has positive characteristic then $L$ has a locally nilpotent ideal of finite codimension.
\end{enumerate}
\end{thm}

\begin{proof}
By Theorem \ref{main}, $L$ has an ideal $I$ of finite codimension satisfying the identity $E_d=0$.
The conclusion now follows from well known theorems of Zelmanov on the nilpotency of $d$-Engel Lie algebras in
characteristic zero \cite{Z0} and the local nilpotency of $d$-Engel Lie algebras in positive characteristic
\cite{Z1, Z2}.
\end{proof}

\section{Algebras over finite fields}

From now on, we assume $k=\F_q$.
For every positive integer $d$, we write $d = m(q-1)+r$, where $0\le r \le q-2$, and
define
$$f_q(d) = \frac{q-r}{q^{m+1}}.$$

\begin{lem}
\label{estimate}
We have
$$f_q(d) = \min\prod_{i=1}^\infty \frac{q-x_i}q,$$
where the sum ranges over all infinite real sequences $x_1,x_2,x_3,\ldots\in [0,q-1]$ summing to $d$.
\end{lem}

\begin{proof}
As $\log (q-x)$ is concave, the value of the product can only decrease if we replace
the sequence with
$$x_1,\ldots,x_{n-1},\sum_{i=n}^\infty x_i,0,0,\ldots$$
Thus, we may consider only sequences which are eventually zero.  If any two non-zero terms $x_i$ and $x_j$ satisfy $x_i+x_j\le q-1$, then we can decrease the product by replacing $x_i$ and $x_j$ by $x_i+x_j$ and $0$ respectively, so we may assume any two non-zero terms sum to more than $q-1$.
If $q-1 > x_i > x_j > 0$ and $x_i+x_j > q-1$, then we can decrease the product by replacing
$x_i$ and $x_j$ by $q-1$ and $x_i+x_j-(q-1)$ respectively.
Thus, we may assume that there is at most one $x_i$ which is neither $0$ nor $q-1$.
Without loss of generality, the sequence can be taken to be non-increasing, so the minimum is achieved
for
$$x_1=\cdots=x_m=q-1,\,x_{m+1}=r, \,x_{m+2} = \cdots = 0.$$

\end{proof}

\begin{lem}
For $1\le k\le q-1$ and $d\ge k$,
$$f_q(d) \le \frac{q-k}q f_q(d-k).$$
\end{lem}

\begin{proof}
This follows immediately from the description of $f_q$ in Lemma~\ref{estimate}.
\end{proof}

\begin{lem}
For $q\ge 2$, we have $f_q(d) \ge 2^{-d}$.
\end{lem}

\begin{proof}
We have $f_q(0)=1$, and writing $d=m(q-1)+r$, with $0\le r\le q-2$, we have
$$\frac{f_q(d+1)}{f_q(d)} = \frac{q-r-1}{q-r} \ge \frac 12.$$
\end{proof}

\begin{thm}
\label{bounded degree}
Let $\F_q$ be a finite field and $P(x_1,\ldots,x_n) \in \F_q[x_1,\ldots,x_n]$ a polynomial of degree $d$.
If
$$N_P :=|\{(a_1,\ldots,a_n)\in \F_q^n\mid P(a_1,\ldots,a_n)\neq 0\}|$$
is non-zero, it satisfies
$$N_P \ge f_q(d)q^n \ge 2^{-d}q^n.$$
\end{thm}

\begin{proof}

As a function on $\F_q^n$, $P$ depends only on its residue class modulo $(x_1^q-x_1,\ldots,x_n^q-x_n)$.  Each such residue class contains a unique element which is of degree $<q$ in each variable separately, and the (total) degree of this representative achieves the minimal degree of all polynomials in the residue class of $P$.
As $f_q(d)$ decreases monotonically in $d$,
we may assume that $P$ is of degree less than $q$ in each variable separately.

We use induction on $d$, the base case $d=0$ being trivial.  Without loss of generality, we may assume that, as a function on $\F_q^n$, $P$ is not constant in the variable $x_n$.  Since the $x_n$-degree of $P$ is less than $q$,
this means that the $x_n$-degree is $l \in [1,q-1]$.  We write
$$P(x_1,\ldots,x_n) = \sum_{i=0}^l P_i(x_1,\ldots,x_{n-1}) x_n^i.$$

As $l +\deg P_l = \deg x_n^l P_l \le \deg P$, we have $\deg P_l \le d-l$.
If $(a_1,\ldots,a_{n-1})\in \F_q^{n-1}$ satisfies $P_l(a_1,\ldots,a_{l-1})\neq 0$, there are at least
$q-l$ solutions of $P(a_1,\ldots,a_{n-1},x_n)\neq 0$.  By the induction hypothesis,
$$N_P\ge (q-l) N_{P_l} \ge
(q-l)f_q(d-l)q^{n-1}=\frac{q-l}q f_q(d-l) q^n\ge f_q(d)q^n.$$

\end{proof}

We can now prove Theorem \ref{dixon}.

\begin{proof}

We have to show that, if the evaluation map $e_Q$ associated with $Q$ is not identically $0$, then $\frac{|e_Q^{-1}(0)|}{|A|^n}< 1 - 2^{-d}$.

If $e_Q$ does not vanish on $A$, then there exists a linear functional $f\colon A\to \F_q$ such that $f\circ e_Q$ does not vanish on $A^n$.  If $Q$ is of degree $d$, then $f\circ e_Q$ has degree $\le d$.  Theorem \ref{dixon} now follows from Theorem~\ref{bounded degree}.
\end{proof}

Let $n$ be a positive integer and consider
$\F_q\langle x_1,\ldots,x_n\rangle$, the algebra over $\F_q$ of the free magma on $n$ generators.  This is a graded $\F_q$-algebra.  Let $Q\in \F_q\langle x_1,\ldots,x_n\rangle$ denote a non-zero element of degree $d$.  If $I$ is any ideal of $A$ of finite codimension, then $Q$ induces a map $(A/I)^n\to A/I$, which we denote $Q_I$ .
Let
$$f(Q,I) := \frac{|Q_I^{-1}(0)|}{|A/I|^n}.$$
Regarding $A/I$ as a finite-dimensional vector space, $Q_I$ is given by a polynomial of degree $\le d$, so either it maps $(A/I)^n$ to $0$, or
$$f(Q,I) \le 1-2^{-d}.$$

We now prove Theorem~\ref{finite fields}.

\begin{proof}
We first prove that condition (1) implies condition (2). We have to show that if $Q$ is not a coset identity
of $A$ then for all $\epsilon > 0$, there exists an ideal $I$ of finite codimension such that $f(Q,I)\le \epsilon$.
We first prove that it implies that for any ideal $I$ of finite codimension, there exists an ideal $J\subset I$ of finite codimension such that
$$f(Q,J) \le (1-2^{-d}) f(Q,I).$$
For each element $\alpha\in Q_I^{-1}(0)$, we choose a representative
$\tilde\alpha = (a_1,\ldots,a_n) \in A^n$ such that $e_Q(a_1,\ldots,a_n)\neq 0$ and an ideal of finite codimension $I_\alpha$ to which
$e_Q(a_1,\ldots,a_n)\neq 0$ does not belong.  Let
$$J = I \cap \bigcap_\alpha I_\alpha,$$
which, by construction, is of finite codimension.
Again by construction, $Q_J$ does not map any $n$-tuple of cosets of $I/J$ to $0$.  Therefore, for each such $n$-tuple, the number of elements mapping to $0$ by $Q_J$ is at most $(1-2^{-d})|I/J|^n$.
If the coset maps to an element of $(A/I)^n$ which is not in $Q_I^{-1}(0)$, then no element of that coset maps to $0$ by $Q_J$.  This proves the claim, and the equivalence of conditions (1) and (2) follows immediately.

Now, suppose $Q$ is multilinear. We will show that condition (2) implies condition (3) with the same ideal $I$.
Assuming (2) we have
\[
e_Q(a_1 + y_1, \ldots , a_n + y_n)=0
\]
for all $y_1, \ldots , y_n \in I$.
By the multilinearity of $Q$ we have, for all $y_1 \in I$,
\begin{align*}
0 = e_Q(a_1 + y_1, a_2, \ldots , a_n) &= e_Q(a_1, a_2, \ldots , a_n) + e_Q(y_1, a_2, \ldots , a_n) \\
&= e_Q(y_1, a_2, \ldots , a_n).
\end{align*}
Similarly we have, for all $y_2 \in I$,
\[
0 = e_Q(y_1, a_2 + y_2, a_3, \ldots , a_n) = e_Q(y_1, a_2, a_3, \ldots , a_n) + e_Q(y_1, y_2, a_3, \ldots , a_n).
\]
Proceeding in this way we obtain
\[
e_Q(y_1, \ldots , y_n)= 0
\]
for all $y_1, \ldots , y_n \in I$.

This completes the proof.

\end{proof}


\begin{thebibliography}{Stacks}

\bibitem[A]{A} Amitsur, Shimshon: An embedding of PI-rings, {\it Proc. Amer. Math. Soc.} {\bf 3} (1952), 3--9.

\bibitem[B]{B} Bourbaki, Nicolas:
\'El\'ements de math\'ematique. Alg\`ebre. Chapitres 4 \`a 7.
Masson, Paris, 1981.

\bibitem[D]{D} Dixon, John: Probabilistic group theory, {\it C.R. Math. Acad. Sci. Soc. R. Can.} {\bf 24}
(2002), 1--15.

\bibitem[DF]{DF} Drensky, Vesselin; Formanek, Edward: Polynomial identity rings. Advanced Courses in Mathematics, CRM Barcelona,
Birkh{\" a}user Verlag, Basel, 2004.

\bibitem[LS]{LS}
Larsen, Michael; Shalev, Aner:
Almost PI algebras are PI, arXiv::1910.05764, to appear \emph{Proc.\ Amer.\ Math.\ Soc.}

\bibitem[MM]{MM} Mann, Avinoam; Martinez, Consuelo: Groups nearly of prime exponent and nearly Engel Lie algebras,
{\it Arch. Math. (Basel)} {\bf 71} (1998), no. 1, 5-–11.


\bibitem[Stacks]{Stacks} Stacks Project, \texttt{https://stacks.math.columbia.edu}.

\bibitem[Z1]{Z0} Zelmanov, Efim: Engel Lie algebras. (Russian) \emph{Dokl. Akad. Nauk SSSR} \textbf{292} (1987), no. 2, 265-–268.

\bibitem[Z2]{Z1} Zelmanov, Efim: Solution of the restricted Burnside problem for groups of odd exponent,
(Russian) \emph{Izv. Akad. SSSR Ser. Mat.} \textbf{54} (1990), 42-–59, 1990; translation in \emph{Math
USSR-Izv.} \textbf{36} (1990), 41–-60.

\bibitem[Z3]{Z2} Zelmanov, Efim: Solution of the restricted Burnside problem for 2-groups, (Russian)
\emph{Mat. Sb.} \textbf{182} (1991), 568-–592; translation in \emph{Math USSR-Sb} \textbf{72} (1992), 543-–565.



\end{thebibliography}
\end{document}